\documentclass[a4paper,reqno]{amsart}
\usepackage{amssymb, amsmath, amscd}
\usepackage{color}
\date{26 June 2013 version 6}
\def\XZERO{\textstyle\frac12}
\def\XONE{\textstyle\frac12}
\def\XTWO{\textstyle\frac \alpha{4(1-\alpha)}}
\def\XTHREE{\textstyle\frac1{2(\alpha-1)}}
\def\XFOUR{\textstyle\frac12}
\def\XFIVE{\textstyle\frac{\alpha-1}{4(2-\alpha)}}
\def\XSIX{\textstyle\frac{3-\alpha}{4(1-\alpha)(2-\alpha)}}
\def\XSEVEN{\textstyle\frac1{(1-\alpha)(2-\alpha)}}
\def\XEIGHT{-\textstyle\frac{\alpha+1}{4(2-\alpha)(1-\alpha)}}
\def\XNINE{-\textstyle\frac{1-\alpha}{8(2-\alpha)}}
\def\XTEN{\textstyle\frac{3\alpha-5}{16(2-\alpha)}}
\def\XELEVEN{-\textstyle\frac{1-\alpha}{8(2-\alpha)}}
\def\XTWELVE{-\textstyle\frac{3-\alpha}{4(1-\alpha)(2-\alpha)}}
\def\XTHIRTEEN{0}
\def\XFOURTEEN{\textstyle\frac1{2(\alpha-2)}}
\def\YZERO{\textstyle\frac12}
\def\YONE{\textstyle\frac12}
\def\YTWO{0}
\def\YTHREE{-\textstyle\frac1{2}}
\def\YFOUR{\textstyle\frac12}
\def\YFIVE{-\textstyle\frac{1}{8}}
\def\YSIX{\textstyle\frac{3}{8}}
\def\YSEVEN{\textstyle\frac1{2}}
\def\YEIGHT{-\textstyle\frac{1}{8}}
\def\YNINE{-\textstyle\frac{1}{16}}
\def\YTEN{-\textstyle\frac{5}{32}}
\def\YELEVEN{-\textstyle\frac{1}{16}}
\def\YTWELVE{-\textstyle\frac{3}{8}}
\def\YTHIRTEEN{0}
\def\YFOURTEEN{-\textstyle\frac1{4}}

\def\pbg{\vphantom{\vrule height 11pt}}
\newtheorem{theorem}{Theorem}[section]

\newtheorem{lemma}[theorem]{Lemma}
\newtheorem{remark}[theorem]{Remark}
\newtheorem{corollary}[theorem]{Corollary}

\makeatletter
 \@addtoreset{equation}{section}
\makeatother
\begin{document}
\title[Heat flow]
{Heat flow out of a compact {manifold}}
\author{M. van den Berg and P. Gilkey}
\address
{School of Mathematics, University of Bristol\\
University Walk, Bristol BS8 1TW\\United Kingdom}
\begin{email}{mamvdb@bris.ac.uk}\end{email}
\address{Mathematics Department, University of Oregon\\
Eugene, OR 97403\\USA}
\begin{email}{gilkey@uoregon.edu}\end{email}
\begin{abstract} We discuss the heat content asymptotics associated with the heat flow
out of a smooth compact manifold in a larger compact Riemannian
manifold. Although there are no boundary conditions, the
corresponding heat content asymptotics involve terms localized
on the boundary. The classical pseudo-differential calculus is
used to establish the existence of the complete asymptotic series
and methods of invariance theory are used to determine the first
few terms in the asymptotic series in terms of geometric data. The
operator driving
the heat process is assumed to be an operator of Laplace type.\\
{Subject classification: 58J32; 58J35; 35K20}
 \end{abstract}

\maketitle

\section{Introduction}
\subsection{Basic assumptions}\label{sect-1.1}  We adopt the following notational conventions.
We let $(M,g)$ be a smooth Riemannian
manifold of dimension $m$.
Let $\Omega\subset M$ be a smooth compact submanifold of $M$ which has dimension
$m$ as well.
We suppose that the boundary of $\Omega$ is non-empty and smooth. In this paper,
we shall investigate the heat flow out of $\Omega$ into $M\setminus\Omega$ where $(M,g)$
satisfies exactly one of the following three conditions:
\begin{enumerate}
\item $M$ is a compact and without boundary.
\item $(M,g)=(\mathbb{R}^m,g_e)$ where $g_e$ is the usual Euclidean inner metric
on $\mathbb{R}^m$.
\item $M$ is a compact submanifold of $\mathbb{R}^m$ with smooth boundary and $g=g_e|_M$.
\end{enumerate}
The case of flat space is fundamental in this subject and hence we have concentrated on that
setting in (2) and (3) to avoid technical difficulties. However, the results of this paper hold more generally.
\subsection{Heat content}
Let $\Delta_g$ be
the associated Laplace-Beltrami operator acting on smooth
functions on $M$. The heat equation on $M$ takes the form
\begin{equation}\label{eqn-1.a}
\Delta_gu+\frac{\partial u}{\partial t}=0,\ x\in M,\ t>0,
\end{equation}
with initial condition
\begin{equation}\label{eqn-1.b}
u(x;0)=\phi(x),\ x\in M,
\end{equation}
where $\phi:M\rightarrow \mathbb R$ is continuous. Here $u(x;t)$
represents the temperature at a point $x\in M$ at time $t$ if $M$
has initial temperature profile $\phi$. We say that $K(x,\tilde x;t)$ is a {\it fundamental solution}
if we have:
$$
u(x;t)=\int_MK(x,\tilde \phi(\tilde x)d\tilde x\,.
$$
We say that $(M,g)$ is {\it stochastically complete} if
$$\int_MK(x,\tilde x;t)d\tilde x=1\text{ for all }x\in M,\ \  t>0\,.$$

In order to guarantee that
 Equation~\eqref{eqn-1.a} and Equation~\eqref{eqn-1.b} have a unique solution,
we have to impose some conditions on the geometry of $M$. These
are summarized in the following result (see Chapter VII of \cite{C} and the references therein):

\goodbreak\begin{theorem}\label{thm-1.1}
\ \begin{enumerate}
\item If $(M,g)$ is  compact without boundary or if $(M,g)$ is
complete and with Ricci curvature bounded from below, then there
exists a unique minimal positive fundamental solution $K$.
Moreover $(M,g)$ is stochastically complete.
\item If $(M,g)$ is compact with non-empty smooth boundary $\partial M$, then there exists a unique minimal
positive fundamental solution $K$. In this case $K$ is the
Dirichlet heat kernel for $M$; $u(x;t)=0$
for all $x\in \partial M$ and for all $t>0$.
\end{enumerate}
\end{theorem}

Let $\rho:M \rightarrow \mathbb R$ be the specific heat of $M$; we
suppose that $\rho$ is continuous. Following \cite{B12} we
define the {\it heat content of $\Omega$ in $M$} by setting:
\begin{equation}\label{eqn-1.c}
\beta_\Omega(\phi,\rho,\Delta_g)(t):=\int_{\Omega}\int_{\Omega}K(x,\tilde
x;t)\phi(x)\rho(\tilde x)dxd\tilde x.
\end{equation}
Let $\phi$ and $\rho$ be smooth on $\Omega$ and let $(M,g)$ satisfy one of
the conditions described in Section~\ref{sect-1.1}.
Theorem~\ref{thm-1.2} and Theorem~\ref{thm-1.3} below show there exists a
complete asymptotic series
\begin{eqnarray*}
\beta_\Omega(\phi,\rho,\Delta_g)(t)&\sim&\sum_{n=0}^\infty
t^{n}\beta_{n}^\Omega(\phi,\rho,\Delta_g) +\sum_{j=0}^\infty
t^{(1+j)/2}\beta_{j}^{\partial\Omega}(\phi,\rho,\Delta_g)\text{ as }t\rightarrow0^+,
\end{eqnarray*}
where $\beta_{n}^\Omega$ and $\beta_{j}^{\partial\Omega}$ are the
integrals of certain locally computable invariants over $\Omega$
and $\partial\Omega$, respectively. We extend $\phi$ by $0$ on
$M\setminus \Omega$, and note that $\phi$ may be discontinuous on
all or part of $\partial \Omega$. In that case the initial
condition is satisfied for $x\in M\setminus \partial \Omega$.

The study of the {\it heat content of $\Omega$ in $\mathbb
R^m$} was initiated in \cite{P}. It was shown \cite{MPPP1,MPPP2,P}
that if $\Omega$ is an open set in $\mathbb R^m$ with finite
Lebesgue measure $|\Omega|$ and with finite perimeter $\mathcal
P(\Omega)$ then
$$
\mathcal P(\Omega)=\lim_{t\rightarrow
0}\left(\frac{\pi}{t}\right)^{1/2}\iint_{\Omega \times (\mathbb
R^m\setminus\Omega)}K(x,\tilde x;t)dxd\tilde x,
$$
where $K(x,\tilde x;t)=(4\pi t)^{-m/2}e^{-|x-\tilde x|^2/(4t)}$ is
the heat kernel for $\mathbb R^m$. It immediately follows that
\begin{equation}\label{eqn-1.d}
\beta_\Omega(1,1,\Delta_{g_e})(t)=|\Omega|-\pi^{-1/2}\mathcal
P(\Omega)t^{1/2}+o(t^{1/2}), \ t\downarrow 0,
\end{equation}
where $|\Omega|=\int_{\Omega}1dx.$ We note that the main
contribution beyond the constant term $|\Omega|$ in
Equation~\eqref{eqn-1.d} comes from localization near $\partial \Omega$.
We shall see presently in Theorem~\ref{thm-5.1} that only the geometry
near $\Omega$ plays a role in the heat
content modulo an exponentially small error in $t^{-1}$ as
$t\downarrow0$ and which therefore plays no role in the asymptotic
series. Even though the $t\downarrow 0$ behaviour of the heat
kernel is known for general $(M,g)$ \cite{YK}, this explicit
asymptotic behaviour does not give much insight, and is not
helpful in the determination of the locally computable invariants of
$\Omega$ and of $\partial \Omega$ respectively.

We must employ a more general formalism; even if we were only interested
in the scalar Laplacian, it is a facet of the ``method of universal examples" that one
must examine this more general framework. Let $(M,g)$ be as described
in Section~\ref{sect-1.1}. Let $D_M$ be an
operator of Laplace type on a smooth vector bundle $V$ over a
complete Riemannian manifold $(M,g)$. Let $\Omega$ be a compact
manifold in $M$ with smooth boundary. Let $\phi\in
L^1(V|_{\Omega})$ represent the initial temperature, and let
$\rho\in L^1(V^*|_{\Omega})$ represent the specific heat. As above
we extend $\phi$ and $\rho$ to $M$ to be zero on $\Omega^c$, and
we denote the resulting extensions by $\phi_\Omega$ and
$\rho_\Omega$ to emphasize that they are supported on $\Omega$.
Similarly to Equation~\eqref{eqn-1.c} we define the {\it heat content of
$\Omega$ in $M$} by
$$\beta_\Omega(\phi,\rho,{ D_M})(t)=\beta_M(\phi_{\Omega},\rho_{\Omega},D_M)(t)=\int_\Omega\int_\Omega
\langle K(x,\tilde x;t)\phi(x),\rho(\tilde x)\rangle dxd\tilde
x.$$ where $\langle\cdot,\cdot\rangle$ denotes the natural pairing
between $V$ and $V^*$. As above we shall suppose that both
$\rho$ and $\phi$ are smooth on the interior of $\Omega$. However, we obtain additional
information by permitting $\phi$ to have a controlled singularity
near $\partial\Omega$. If $r$ is the geodesic distance to the
boundary, we shall suppose that $r^{\alpha}\phi$ is smooth near
the boundary for some fixed complex number $\alpha$. We shall
always assume $\Re(\alpha)<1$ to ensure that $\phi\in
L^1(\Omega)$. The parameter $\alpha$ controls the blow up (if
$\Re(\alpha)>0$) or decay (if $\Re(\alpha)<0$) of $\phi$ near the
boundary.
We permit $\alpha$ to be complex. Although this has
no physical significance, it is useful in analytic continuation
arguments as we shall see in the proof of Lemma~\ref{lem-3.1}. Let
$\mathcal{K}_\alpha=\mathcal{K}_\alpha(V)$ denote the resulting
space of functions. We will establish the following result in Section~\ref{sect-2}:
\begin{theorem}\label{thm-1.2}
Adopt the notation established above. Let $M$ be a compact Riemannian
manifold without boundary. There is a complete
asymptotic series as $t\downarrow0$ of the form:
\begin{eqnarray*}
\beta_\Omega(\phi,\rho,D_M)(t)&\sim&\sum_{n=0}^\infty
t^{n}\beta_{n}^\Omega(\phi,\rho,D_M) +\sum_{j=0}^\infty
t^{(1+j-\alpha)/2}\beta_{j,\alpha}^{\partial\Omega}(\phi,\rho,D_M),
\end{eqnarray*}
where $\beta_{n}^\Omega$ and $\beta_{j,\alpha}^{\partial\Omega}$
are integrals of certain locally computable invariants over
$\Omega$ and $\partial\Omega$, respectively.
\end{theorem}

In Section~\ref{sect-5}, we will use Theorem~\ref{thm-1.2} to establish the following result:
\begin{theorem}\label{thm-1.3}
 Let $(M,g)=(\mathbb{R}^m,g_e)$ or let $M$ be a compact {submanifold} of $\mathbb{R}^m$
 with smooth boundary and $g=g_e|_M$. There is a complete
asymptotic series as $t\downarrow0$ of the form:
\begin{eqnarray*}
\beta_\Omega(\phi,\rho,D_M)(t)&\sim&\sum_{n=0}^\infty
t^{n}\beta_{n}^\Omega(\phi,\rho,D_M) +\sum_{j=0}^\infty
t^{(1+j-\alpha)/2}\beta_{j,\alpha}^{\partial\Omega}(\phi,\rho,D_M),
\end{eqnarray*}
where $\beta_{n}^\Omega$ and $\beta_{j,\alpha}^{\partial\Omega}$
are integrals of certain locally computable invariants over
$\Omega$ and $\partial\Omega$, respectively.
\end{theorem}

\subsection{1-dimensional geometry} What happens on the line is in many ways crucial to our analysis as we shall
``bootstrap" our way from that setting to the higher dimensional setting. Let $S^1=[0,2\pi]$ where we identify $0\sim2\pi$.
Let $\Omega=[0,\pi]\subset S^1$. Let $V=V^*$ be the trivial bundles and let $D=-\partial_x^2$.
Near the boundary point $x=0$, we may expand $\phi$ and $\rho$ in modified Taylor series for $x>0$:
$$\phi(x)\sim x^{-\alpha}\{\phi_0+\phi_1x+\phi_2x^2+...\}\text{ and }\rho(x)\sim\rho_0+\rho_1x+\rho_2x^2+... .$$
There is a similar expansion near $x=\pi$ as well; $\partial\Omega=\{0,\pi\}$ and the boundary $dy$
is simply counting measure in this instance.
We will establish the following result in Section~\ref{sect-3}:

\begin{lemma}\label{lem-1.4}
With the notation established above we may take
$$\beta_n^\Omega=\frac{(-1)^n}{n!}\int_\Omega \phi\cdot D^n\rho
dx,\text{ and }\beta_{j,\alpha}^{\partial\Omega}
=\sum_{k+\ell=j}\int_{\partial\Omega}c_{k,\ell,\alpha}\phi_k\rho_\ell
dy,$$ where $dy$ is the Riemannian volume element on the boundary
of $\partial \Omega$, and where
$$
c_{k,\ell,\alpha}=(-1)^{\ell+1}\frac1{\sqrt{4\pi}}\int_{0}^\infty\int_{0}^\infty e^{-(u+\tilde u)^2/4}u^{k-\alpha}\tilde u^\ell
dud\tilde u\ .$$
\end{lemma}

\subsection{Local invariants}
We return to the general setting.
By considering the case in which $\Omega=M$,
we see that $\beta_n^\Omega$ can be taken to be
$$\beta_{n}^\Omega(\phi,\rho,{ D_M})=\frac{(-1)^n}{n!}\int_\Omega\langle\phi,
\tilde D_M^n\rho\rangle dx,$$ where $\tilde D_M^n$ is the dual
operator on $V^*$. We shall not differentiate $\phi$ on the
interior owing to the lack of smoothness in $\phi$ as we approach
the boundary. This is a crucial point: were we to examine the
doubly singular setting, we would need to regularize the interior
integrals. Thus attention is focused on the boundary invariants
$\beta_{j,\alpha}^{\partial\Omega}$; these are only defined up to
divergence terms.

\subsection{A Bochner formalism}
Although in Sections~\ref{sect-2} and \ref{sect-3} we will use a
coordinate formalism for describing the invariants
$\beta_{j,\alpha}^{\partial\Omega}$, it is useful to introduce an
invariant tensorial formalism at this point. There is a unique
connection $\nabla$ on $V$ and a unique endomorphism $E$ of $V$ so
that we may express $ D_M$ using a Bochner formalism:
$$ D_M\phi=-\{g^{ij}\phi_{;ij}+E\phi\}\,.$$
Here $\phi_{;ij}$ denotes the components of the second covariant
derivative $\nabla^2\phi$ of $\phi$. We adopt the {\it Einstein}
convention and sum over repeated indices (see Gilkey \cite{G94}
for details). Let $\Gamma$ denote the Christoffel symbols of the
metric $g$. If we express $ D_M$ in a coordinate system in the
form:
\begin{equation}\label{eqn-1.e}
\displaystyle  D_M=-\{ g^{ij}\partial_{x_i}
\partial_{x_j}+A^i\partial_{x_i}+B\},
\end{equation}
then the connection $1$-form $\omega$ of
$\nabla$ and the endomorphism $E$ are given by
$$
\omega_i=\textstyle{\textstyle\frac12}(g_{ij}A^j
+g^{kl}\Gamma_{kli}\operatorname{Id})\text{ and }
E=B-g^{ij}(\partial_{x_j}\omega_{i}+\omega_{i}\omega_{j}
 -\omega_{k}\Gamma_{ij}{}^k)\,.
$$

The question always arises as to why it is necessary (or desirable) to consider these more
general operators of Laplace type when in practice one is usually only interested in the scalar
Laplacian. In Section~\ref{sect-3.4}, we will
evaluate the coefficients of the terms involving the second fundamental form $L_{ab}$
and the Ricci curvature $\operatorname{Ric}_{mm}$ in $\beta_{1,\alpha}^{\partial\Omega}$ and
$\beta_{2,\alpha}^{\partial\Omega}$. We will use warped products and operators which are
not the scalar Laplacian. The fact that we are working with quite general operators
will be crucial. It is typical in this subject that to obtain formulas for the scalar Laplacian,
one must invoke the more general formalism and derive formulas in a very general context.

Let $\nabla^*$ be the dual connection on $V^*$, and let the
associated connection $1$-form be given by $-\omega_i^*$. We let
indices $\{i,j,k,l\}$ range from $1$ to $m$. Let $R_{ijkl}$ be the
components of the metric tensor,
 let $\operatorname{Ric}_{ij}$ be the components of the Ricci tensor, and let $\tau$ be the scalar curvature.
Let indices $\{a,b,c\}$ range from $1$ to $m-1$. Near the boundary, we choose a local orthonormal frame
$\{e_i\}$ for the tangent bundle so that $e_m$ is the inward unit geodesic
normal.
Let $L_{ab}$ denote the components of the second fundamental form. Let
$$\phi_j:=\nabla_{e_m}^j\phi|_{\partial M}\text{ and }
 \rho_j:=(\nabla^*_{e_m})^j\rho|_{\partial M}\,.$$
One may use dimensional analysis to see that
$\beta_{j,\alpha}^{\partial\Omega}$ is homogeneous of order $j$ in
the derivatives of the structures involved. This leads to the
following observation.
\begin{lemma}\label{lem-1.5}
There exist universal constants
$\varepsilon_{\nu,\alpha}$ which depend holomorphically on $\alpha$ so that:
\medbreak
$\beta_{0,\alpha}^{\partial\Omega}
 (\phi,\rho, D_M)=\displaystyle\int_{\partial\Omega}
 \varepsilon_{0,\alpha}\langle\phi_0,\rho_0\rangle dy$.
\medbreak$ \beta_{1,\alpha}^{\partial\Omega}(\phi,\rho,
D_M)=\displaystyle\int_{\partial\Omega}\left\{\varepsilon_{1,\alpha}\langle\phi_1,\rho_0\rangle
+\varepsilon_{2,\alpha}\langle
L_{aa}\phi_0,\rho_0\rangle+\varepsilon_{3,\alpha}\langle\phi_0,\rho_1\rangle\right\}
dy$. \medbreak $\beta_{2,\alpha}^{\partial\Omega}(\phi,\rho,
D_M)=\displaystyle\int_{\partial\Omega}\{\varepsilon_{4,\alpha}\langle\phi_2,\rho_0\rangle
+\varepsilon_{5,\alpha}\langle
L_{aa}\phi_1,\rho_0\rangle+\varepsilon_{6,\alpha}\langle
E\phi_0,\rho_0\rangle$ \medbreak\quad
$+\varepsilon_{7,\alpha}\langle\phi_0,\rho_2\rangle
+\varepsilon_{8,\alpha}\langle
L_{aa}\phi_0,\rho_1\rangle+\varepsilon_{9,\alpha}\langle\operatorname{Ric}_{mm}\phi_0,\rho_0\rangle
+\varepsilon_{10,\alpha}\langle L_{aa}L_{bb}\phi_0,\rho_0\rangle$
\medbreak\quad $+\varepsilon_{11,\alpha}\langle
L_{ab}L_{ab}\phi_0,\rho_0\rangle +\varepsilon_{12,\alpha}\langle
\phi_{0;a},\rho_{0;a}\rangle+\varepsilon_{13,\alpha}\langle\tau\phi_0,\rho_0\rangle
+\varepsilon_{14,\alpha}\langle\phi_1,\rho_1\rangle\}dy$.
\end{lemma}
We omit details in the interests of brevity and refer
instead to the discussion in \cite{BGS08}, where a similar result
is established for the heat content asymptotics which arise from
the Dirichlet realization.

Although we have in principle permitted vector valued operators, this lack of commutativity plays no
role at the $\beta_{2,\alpha}^{\partial\Omega}$ level and thus we restrict henceforth to scalar operators.

\subsection{Normalizing constants}
Following the discussion in \cite{BGS08}, we define:
$$
c_\alpha:=2^{1-\alpha}\Gamma\left(\frac{2-\alpha}2\right)\frac1{\sqrt\pi(\alpha-1)}\,.
$$
Because
$s\Gamma(s)=\Gamma(s+1)$, we have the recursion relations
\begin{equation}\label{eqn-1.f}
c_\alpha=-\frac{\alpha-3}{2(\alpha-1)(\alpha-2)}c_{\alpha-2}\quad\text{and}\quad
c_{\alpha+1}=-\frac{\alpha-2}{2\alpha(\alpha-1)}c_{\alpha-1}\,.
\end{equation}
We set $\alpha=0$ to see:
$$
c_0=-\frac2{\sqrt\pi},\quad c_{-1}=-1,\quad c_{-2}=-\frac8{3\sqrt\pi}\,.
$$

\subsection{Some terms in the asymptotic series} In Section~\ref{sect-3}, we
will establish the following result by determining the universal
constants in Lemma~\ref{lem-1.5}. This together with
Theorem~\ref{thm-1.2} are the two main results of this paper.

\begin{theorem}\label{thm-1.6}
\ \medbreak
$\beta_{0,\alpha}^{\partial\Omega}
 (\phi,\rho,{D_M})=c_\alpha\textstyle\int_{\partial\Omega}
 \XZERO\langle\phi_0,\rho_0\rangle dy$.
\medbreak$ \beta_{1,\alpha}^{\partial\Omega} (\phi,\rho,D_M) =\textstyle
c_{\alpha-1}\int_{\partial\Omega}\left\{\XONE\langle\phi_1,\rho_0\rangle
+\XTWO\langle L_{aa}\phi_0,\rho_0\rangle +\XTHREE
\langle\phi_0,\rho_1\rangle\right\} dy. $\medbreak$
\beta_{2,\alpha}^{\partial\Omega} (\phi,\rho,{D_M})
 =\textstyle c_{\alpha-2}\int_{\partial\Omega}\{\XFOUR
 \langle\phi_2,\rho_0\rangle
+\XFIVE\langle
L_{aa}\phi_1,\rho_0\rangle
 +\XSIX\langle E
 \phi_0,\rho_0\rangle$
 \medbreak\quad
$+\XSEVEN\langle
\phi_0,\rho_2\rangle$
$\XEIGHT\langle
 L_{aa}\phi_0,\rho_1\rangle
 \XNINE\langle
 \operatorname{Ric}_{mm}\phi_0,\rho_0\rangle$
 \medbreak\quad
$ +\XTEN\langle
 L_{aa}L_{bb}\phi_0,\rho_0\rangle
 \XELEVEN\langle
L_{ab}L_{ab}\phi_0,\rho_0\rangle$
 \medbreak\quad
$
\XTWELVE
\langle \phi_{0;a},\rho_{0;a}\rangle
+\XTHIRTEEN
\langle\tau\phi_0,\rho_0\rangle+
\XFOURTEEN
\langle\phi_1,\rho_1\rangle\}dy$.
\end{theorem}

We specialize Theorem~\ref{thm-1.6} to the smooth setting by setting $\alpha=0$ to obtain:

\begin{corollary}
\ \medbreak
$\beta_{0}^{\partial\Omega}
 (\phi,\rho, D_M)=-\frac2{\sqrt\pi}\textstyle\int_{\partial\Omega}
 \YZERO\langle\phi_0,\rho_0\rangle dy$.
\medbreak$ \beta_{1}^{\partial\Omega} (\phi,\rho, D_M)
=\textstyle
-\int_{\partial\Omega}\left\{\YONE\langle\phi_1,\rho_0\rangle
+\YTWO\langle L_{aa}\phi_0,\rho_0\rangle \YTHREE
\langle\phi_0,\rho_1\rangle\right\} dy. $\medbreak$
\beta_{2}^{\partial\Omega} (\phi,\rho, D_M)
 =\textstyle -\frac8{3\sqrt\pi}\int_{\partial\Omega}\{\YFOUR
 \langle\phi_2,\rho_0\rangle
\YFIVE\langle
L_{aa}\phi_1,\rho_0\rangle
 +\YSIX\langle E
 \phi_0,\rho_0\rangle$
 \medbreak\quad
$+\YSEVEN\langle
\phi_0,\rho_2\rangle$
$\YEIGHT\langle
 L_{aa}\phi_0,\rho_1\rangle
 \YNINE\langle
 \operatorname{Ric}_{mm}\phi_0,\rho_0\rangle
 \YTEN\langle
 L_{aa}L_{bb}\phi_0,\rho_0\rangle$
 \medbreak\quad
$
 \YELEVEN\langle
L_{ab}L_{ab}\phi_0,\rho_0\rangle
\YTWELVE
\langle \phi_{0;a},\rho_{0;a}\rangle
+\YTHIRTEEN
\langle\tau\phi_0,\rho_0\rangle
\YFOURTEEN
\langle\phi_1,\rho_1\rangle\}dy$.
\end{corollary}

\subsection{Dirichlet and Robin boundary conditions}
Let $\mathcal{B}_Df:=f|_{\partial\Omega}$ be the {\it Dirichlet
boundary operator} and let
$\mathcal{B}_Sf:=(\nabla_mf+Sf)|_{\partial\Omega}$ be the {\it
Robin boundary operator}; here $S$ is an auxiliary endomorphism of
$V|_{\partial\Omega}$. We use $S^*$ and the dual connection to
define the dual boundary operator $\tilde{\mathcal{B}}$ on
$C^\infty(V^*)$. Let $D_{\mathcal{B}}$ denote the Dirichlet or
Robin realization of $D_M$. In this instance, the ambient manifold
$M$ plays no role and defines:
$$\beta(\phi,\rho, D_M,\mathcal{B})=\int_\Omega\langle e^{-tD_{\mathcal{B}}}\phi,\rho\rangle dx\,.$$
The the existence of an analogous asymptotic series in this setting has been established in \cite{BGS08}.
After adjusting the
notation suitably using Equation~(\ref{eqn-1.f}) from that in \cite{BGS08}, the results of \cite{BGS08} yield:
\begin{theorem}\label{thm-1.8}
\ \begin{enumerate}
\item $\beta_{0,\alpha}^{\partial M}(\phi,\rho,D_M,{B_{\mathcal{D}}})=c_\alpha\int_{\partial\Omega}\langle\phi_0,\rho_0\rangle dy$.
\item $\beta_{1,\alpha}^{\partial
M}(\phi,\rho,D_M,{B_{\mathcal{D}}})=c_{\alpha-1}\int_{\partial\Omega}
     \langle\phi_1-{\textstyle{\frac12}}L_{aa}\phi_0,\rho_0\rangle dy$.
\smallbreak\item $\beta_{2,\alpha}^{\partial
M}(\phi,\rho,D_M,{B_{\mathcal{D}}})
   =c_{\alpha-2}\int_{\partial\Omega}\{\langle\phi_2,\rho_0\rangle-\frac12\langle L_{aa}\phi_1,\rho_0\rangle
$\smallbreak\ \ $
    -\frac{\alpha-3}{2(\alpha-1)(\alpha-2)}\langle E\phi_0,\rho_0\rangle
   +\frac2{(\alpha-1)(\alpha-2)}\langle\phi_0,\rho_2\rangle
   -\frac1{(\alpha-1)(\alpha-2)}\langle L_{aa}\phi_0,\rho_1\rangle
$\smallbreak\ \ $+\frac{\alpha-3}{2(\alpha-1)(\alpha-2)}\langle\phi_{0:a},\rho_{0:a}\rangle
    -\frac{\alpha-1}{4(\alpha-2)}\langle\operatorname{Ric}_{mm}\phi_0,\rho_0\rangle$\smallbreak\ \ $
+\frac{\alpha-1}{8(\alpha-2)}\langle
L_{aa}L_{bb}\phi_0,\rho_0\rangle -\frac{\alpha-1}{4(\alpha-2)}
\langle L_{ab}L_{ab}\phi_0,\rho_0\rangle\}dy$. \smallbreak\item
$\beta_{0,\alpha}^{\partial
M}(\phi,\rho,D_M,{B_{\mathcal{R}}})=0$. \smallbreak\item
$\beta_{1,\alpha}^{\partial M}(\phi,\rho,
D_M,{B_{\mathcal{R}}})=\frac{2\alpha}{2-\alpha}c_{\alpha+1}\int_{\partial
M}\langle\phi_0,\tilde B_{\mathcal{R}}\rho\rangle dy$
\smallbreak\qquad$
=-\frac{2\alpha}{2-\alpha}\cdot\frac{\alpha-2}{2\alpha(\alpha-1)}c_{\alpha-1}\int_{\partial\Omega}
\langle\phi_0,\tilde B_{\mathcal{R}}\rho\rangle
dy$\smallbreak\qquad$
=-\frac1{1-\alpha}c_{\alpha-1}\int_{\partial\Omega}
\langle\phi_0,\tilde B_{\mathcal{R}}\rho\rangle dy$.
\smallbreak\item $\beta_{2,\alpha}^{\partial M}(\phi,\rho,
D_M,{B_{\mathcal{R}}}) =\frac{-2}{3-\alpha}c_\alpha\int_{\partial
M}\langle(1-\alpha) \phi_1+S\phi_0-\frac\alpha2L_{aa}\phi_0,\tilde
B_{\mathcal{R}}\rho\rangle dy$ \smallbreak\qquad
$=-\frac1{(\alpha-1)(\alpha-2)}c_{\alpha-2}\int_{\partial
M}\langle(1-\alpha) \phi_1+S\phi_0-\frac\alpha2L_{aa}\phi_0,\tilde
B_{\mathcal{R}}\rho\rangle dy$.
\end{enumerate}
\end{theorem}

\section{Proof of Theorem~\ref{thm-1.2}}\label{sect-2}
We will follow the discussion of the pseudo-differential calculus
based on the work by Seeley \cite{S66,S69}, and refer to Gilkey
\cite{G94}. Throughout this section, we assume
the ambient Riemannian manifold $(M,g)$ is compact and without boundary.

By using a partition of unity, we may assume that
$\rho$ and $\phi$ are supported within coordinate systems. Since
the kernel of the heat equation decays exponentially in
$t^{-1}$ for $\operatorname{dist}_g(x,\tilde
x)\ge\epsilon>0$, we may assume that $\rho$ and $\phi$ have
support within the same coordinate system. There will, of course,
be three different types of coordinate systems to be considered -
those which touch the boundary of $\Omega$, those which are
contained entirely within the interior of $\Omega$, and those
which are contained in the exterior of $\Omega$; those contained
in the exterior of $\Omega$ play no role as they contribute an
exponentially small error in $t^{-1}$. In
Section~\ref{sect-2.1} we establish notational conventions and
prove a technical result. In Section~\ref{sect-2.2} we review the
notion of a pseudo-differential operator; in
Section~\ref{sect-2.3}, we construct the resolvent, and in
Section~\ref{sect-2.4} we construct an approximation to the kernel
of the heat equation. The new material begins in
Section~\ref{sect-2.5} where we begin the examination of the
 heat content. In Section~\ref{sect-2.6}, we establish the existence of two different kinds of asymptotic
 series. In Section~\ref{sect-2.7}, we use the results of Section~\ref{sect-2.6} to
 discuss coordinate systems contained in the interior of $\Omega$ and Section~\ref{sect-2.8}
 we use the results of Section~\ref{sect-2.6} to study coordinate systems
near the boundary. The fact that $D_M$ is of Laplace type plays a
central role in the discussion.

\subsection{Notational conventions}\label{sect-2.1} Let $x=(x^1,...,x^m)\in\mathbb{R}^m$ be coordinates on an
open set $\mathcal{O}\subset M$.
Let $(x,\xi)$ be the induced coordinate system on the cotangent space $T^*(\mathcal{O})$ where
we expand a $1$-form $\omega\in T^*\mathcal{O}$ in the form:
$$\omega=\xi_idx^i\text{ to define }\xi=(\xi_1,...,\xi_m)\,.$$
We let $x\cdot\xi$ be the natural pairing
$$x\cdot\xi:=x^i\xi_i\,.$$
If $\alpha=(a_1,...,a_m)$ is a multi-index, set
$$\begin{array}{llll}
|\alpha|=a_1+...+a_m,&\alpha!=a_1!...a_m!,&\partial_x^\alpha:=\partial_{x_1}^{a_1}...\partial_{x_m}^{a_m},\\
\phi^{(\alpha)}:=\partial_x^\alpha\phi,&
\rho^{(\alpha)}:=\partial_x^\alpha\rho,&D_x^\alpha:=\sqrt{-1}^{|\alpha|}d_x^\alpha,\vphantom{\vrule height 14pt}\\
d_\xi^\alpha=\partial_{\xi_1}^{a_1}...\partial_{\xi_m}^{a_m},&x^\alpha=x_1^{a_1}...x_m^{a_m},&
\xi^\alpha:=\xi_1^{a_1}...\xi_m^{a_m}.\vphantom{\vrule height 14pt}
\end{array}$$
For example, with these notational conventions, Taylor's theorem becomes
\begin{equation}\label{eqn-2.a}
f(x)=\sum_{|\alpha|\le n}\textstyle\frac1{\alpha!}f^{(\alpha)}(x_0)(x-x_0)^\alpha+O(|x-x_0|^{n+1})\,.
\end{equation}
Let $d\nu_x$, $d\nu_{\tilde x}$, and $d\nu_\xi$ denote Lebesgue measure on $\mathbb{R}^m$. Let $L^2_e$ denote $L^2(\mathbb{R}^m)$
with respect to Lebesgue measure. Let $g(x)=g_{ij}(x)dx^i\circ dx^j$ be a Riemannian metric on $\mathcal{O}$ and define:
$$||\xi||_{g^*(x)}^2:=g^{ij}(x)\xi_i\xi_j\text{ and }||x-\tilde x||_{g(x)}^2:=g_{ij}(x)(x^i-\tilde x^i)(x^j-\tilde x^j)\,.$$
We shall always be restricting to compact $x$ and $\tilde x$ subsets. Let
$(\cdot,\cdot)_e$ and $||_e$ denote the usual Euclidean inner product and norm, respectively:
$$(x,y)_e:=x_1y_1+...+x_my_m\text{ and }
||(x_1,...,x_m)||_e^2:=(x,x)_e=x_1^2+...+x_m^2\,.$$
We have estimates:
$$C_1||\xi||_e^2\le ||\xi||_{g^*(x)}^2\le C_2||\xi||_e^2
\text{ and }C_1||x-\tilde x||_e^2\le ||x-\tilde x||_{g(x)}^2\le C_2||x-\tilde x||_e^2
$$
for some constants $C_1>0$ and $C_2>0$. Let $\theta=\sqrt{g}$; $\theta$
is a symmetric metric so that $\theta_{ij}\theta_{jk}=g_{ik}$. We then have
$$ ||\xi||_{g^*(x)}^2=||\theta^{-1}(x)\xi||_e^2\text{ and }||(x-\tilde x)||_{g(x)}^2=||\theta(x)(x-\tilde x)||_e^2\,.$$

\begin{lemma}\label{lem-2.1} We use $\theta^2=g$ to lower indices and regard $\theta^2(x)(x-\tilde x)$ as a vector
$(X-\tilde X)_i:=g_{ij}(x-\tilde x)^j$. With this identification, we have:
\begin{eqnarray*}
&&||\xi||_{g^*(x)}^2+\sqrt{-1}(x-\tilde x)\cdot\xi/\sqrt t\\
&=&||\xi+\textstyle\frac12\sqrt{-1}(X-\tilde X)/\sqrt
t||_{g^*(x)}^2+||(x-\tilde x)||_{g(x)}^2/(4t)\,.
\end{eqnarray*}
\end{lemma}

\begin{proof} We expand:
\medbreak\quad $||\xi||_{g^*(x)}^2+\sqrt{-1}(x-\tilde
x)\cdot\xi/\sqrt{t}$ \medbreak\qquad $=
(\theta^{-1}(x)\xi,\theta^{-1}(x)\xi)_e+\sqrt{-1}(\theta^{-1}(x)\theta^2(x)(\tilde
x-x)/\sqrt t,\theta^{-1}(x)\xi)_e$ \medbreak\qquad
$=(\theta^{-1}(x)\{\xi+\textstyle\sqrt{-1}\theta^{2}(x)(\tilde
x-x)/\sqrt t\},\theta^{-1}(x)\xi)_e$ \medbreak\quad
$=||\theta^{-1}(x)\{\xi+\frac12\sqrt{-1}\theta^{2}(x)(\tilde
x-x)/\sqrt t\}||_e^2 +||\theta(x)(\tilde x-x)||_e^2/(4t)$
\medbreak\qquad $=||\xi+\frac12\sqrt{-1}(X-\tilde X)/\sqrt
t||_{g^*(x)}^2+||(x-\tilde x)||_{g(x)}^2/(4t)$.
\end{proof}

\subsection{Pseudo-differential operators}\label{sect-2.2}
 If $P$ is a pseudo-differential operator with symbol $p(x,\xi)$, then
$P$ is characterized by following identity for all $\phi\in C_0^\infty(V)$ and $\rho\in C_0^\infty(V^*)$:
\begin{equation}\label{eqn-2.b}
\langle P\phi,\rho\rangle_{L_e^2}=(2\pi)^{-m}
\iiint e^{-\sqrt{-1}(x-\tilde x)\cdot\xi}\langle p(x,\xi)\phi(x),\rho(\tilde x)\rangle d\nu_xd\nu_\xi d\nu_{\tilde x}\,.
\end{equation}
The integrals in question here are iterated integrals - the convergence is not absolute
and the $d\nu_x$ integral has to be performed
before the $d\nu_\xi$ integral. However, if $p(x,\xi)$ decays rapidly enough in $\xi$,
then the integrals are in fact absolutely convergent
and we can interchange the order of integration to see following  \cite[Lemma 1.2.5]{G94}
that $P$ is given by a kernel:
\begin{equation}\label{eqn-2.c}
\begin{array}{l}
\displaystyle\langle P\phi,\rho\rangle_{L^2}
=\int\langle K_P(x,\tilde x)\phi(x),\rho(\tilde x)\rangle d\nu_xd\nu_{\tilde x}
\text{ where}\\
\displaystyle K_P(x,\tilde x):=(2\pi)^{-m}\int e^{-\sqrt{-1}(x-\tilde x)\cdot\xi}p(x,\xi)d\nu_\xi\,.
\end{array}\end{equation}

\subsection{The resolvent}\label{sect-2.3}
Let $D_M$ be an operator of Laplace type on $C^\infty(V)$ over
$M$. In a system of local coordinates $(x^1,...,x^m)$ on an open
subset $\mathcal{O}$ of $M$, we may change notation slightly from
that employed in Equation~(\ref{eqn-1.e}) and expand:
$$D_M=a_2^{ij}(x)D_{x_i}D_{x_j}+a_1^i(x)D_{x_i}+a_0(x)\,.$$
We ensure that Equation~(\ref{eqn-2.b}) defines the operator $D_M$
by defining:
\begin{eqnarray*}
&&p(x,\xi)=p_2(x,\xi)+p_1(x,\xi)+p_0(x,\xi)\text{ where }\\
&&p_2(x,\xi):=|\xi|_{g^*(x)}^2,\quad p_1(x,\xi):=a_1^i(x)\xi_i,\quad p_0(x,\xi):=a_0(x)\,.
\end{eqnarray*}

Let $\mathcal{R}\subset\mathbb{C}$ be the complement of a cone of
angle $\varepsilon_{1,\alpha}$ about the positive real axis and a
ball of radius $\varepsilon_{2,\alpha}^{-1}$ about the origin
where
$\varepsilon_{2,\alpha}=\varepsilon_{2,\alpha}(\varepsilon_{1,\alpha})$
is chosen so that $D_M$ has no eigenvalues in $\mathcal{R}$. We
let $\lambda\in\mathcal{R}$ henceforth. Following the discussion
of \cite[Lemma~1.7.2]{G94}, we define $r_n(x,\xi;\lambda)$ for
$(x,\xi)\in T^*(\mathcal{O})$ and $\lambda\in\mathcal{R}$
inductively by setting:
\begin{equation}\label{eqn-2.d}
\begin{array}{l}
r_0(x,\xi;\lambda):=(|\xi|^2_{g^*(x)}-\lambda)^{-1},\vphantom{A_{A_{A_{A_A}}}}\\
r_n:=-r_0\displaystyle\sum_{|\alpha|+2+j-k=n,j<n}
 {\textstyle\frac1{\alpha!}}\partial_\xi^\alpha p_k\cdot D_x^\alpha r_j\text{ for }n>0\,.
\end{array}\end{equation}
Define
$$\begin{array}{lll}
\operatorname{ord}(\partial_x^\alpha p_2)=|\alpha|,&\operatorname{ord}(\partial_x^\alpha p_1)=|\alpha|+1,&
\operatorname{ord}(\partial_x^\alpha p_0|)=|\alpha|+2,\\
\operatorname{weight}(\lambda)=2,&\operatorname{weight}(\xi)=1.\vphantom{\vrule height 12pt}
\end{array}$$
The following lemma follows immediately by induction from
the recursive definition in Equation~(\ref{eqn-2.d}):
\begin{lemma}\label{lem-2.2}
\ \begin{enumerate}
\item $r_n$ is homogeneous of order $n$ in the derivatives of the symbol of $D_M$.
\item$r_n$ has weight $-n-2$ in $(\xi,\lambda)$.
\item There exist polynomials $r_{n,j,\alpha}(x,D_M)$ for $n\le j\le 3n$ which are
homogeneous of order $n$ in the derivatives of the symbol of $D_M$
so:
$$
r_n(x,\xi;\lambda)=\sum_{2j-|\alpha|=n}r_{n,j,\alpha}(x,D_M)(|\xi|_{g^*(x)}^2-\lambda)^{-j-1}\xi^\alpha\,.
$$
\end{enumerate}\end{lemma}

We use Equation~(\ref{eqn-2.b}) to define the pseudo-differential operator $R_n(\lambda)$ with symbol $r_n$
so that
$$\langle R_n(\lambda)\phi,\rho\rangle_{L^2_e}=(2\pi)^{-m}\iiint e^{-\sqrt{-1}(x-\tilde x)\cdot\xi}\langle r_n(x,\xi;\lambda)\phi(x),
\rho(\tilde x)\rangle d\nu_xd\nu_\xi d\nu_{\tilde x}\,.$$
 Let $||_{-k,k}$ be the norm of a map from the Sobolev space $H_{-k}$ to the Sobolev space $H_k$.
By \cite[Lemma~1.7.3]{G94} we have that if $\lambda\ge\lambda(k)$ and if $n\ge n(k)$, then:
$$
 ||(D_M-\lambda)^{-1}-R_0(\lambda)-...-R_n(\lambda)||_{-k,k}\le C_k(1+|\lambda|)^{-k}\,.
$$

 \subsection{An approximation to the kernel of the heat equation}\label{sect-2.4}
Let $\gamma$ be the boundary of $\mathcal{R}$ oriented suitably. We use the operator valued Riemann integral to define
$$e^{-tD_M}:=\frac1{2\pi\sqrt{-1}}\int_\gamma e^{-t\lambda}(D_M-\lambda)^{-1}d\lambda\,.$$
We use \cite[Lemma~1.7.5]{G94} to see that this is the fundamental solution of the heat equation
and belongs to $\operatorname{Hom}(H_{-k},H_k)$ for any $k$. We now let
\begin{equation}\label{eqn-2.e}
e_n(x,\xi;t)=\frac1{2\pi\sqrt{-1}}\int_\gamma e^{-t\lambda}r_n(x,\xi;\lambda)d\lambda
\end{equation}
define the pseudo-differential operator
\begin{equation}\label{eqn-2.f}
E_n(t,D_M):=\frac1{2\pi\sqrt{-1}}\int_\gamma
e^{-t\lambda}R_n(\lambda)d\lambda\,.
\end{equation}
We use Lemma~\ref{lem-2.2}~(3) and Cauchy's integral formula to rewrite Equation~(\ref{eqn-2.e}) as:
\begin{equation}\label{eqn-2.g}
e_n(x,\xi;t)=\sum_{2j-|\alpha|=n}\frac{t^j}{j!}\xi^\alpha
e^{-t|\xi|_{g^*(x)}^2}r_{n,j,\alpha}(x,D_M)\,.
\end{equation}
We now use Equation~(\ref{eqn-2.c}) and Equation~(\ref{eqn-2.g})
to see the operator $E_n$ of Equation~(\ref{eqn-2.f}) is given by the smooth kernel
\begin{eqnarray}\label{eqn-2.h}
&&K_n(x,\tilde x;t)\\
&&\quad:=\sum_{2j-|\alpha|=n}(2\pi)^{-m}\frac{t^j}{j!}\int_{\mathbb{R}^m}
e^{-t|\xi|_{g^*(x)}^2-\sqrt{-1}(x-\tilde x)\cdot\xi}\xi^\alpha
r_{n,j,\alpha}(x,D_M)d\nu_\xi\,.\nonumber
\end{eqnarray}
Let $||_{C^k}$ denote the $C^k$ norm. Given any $k\in\mathbb{N}$, there exists $n(k)$ so that
if $n\ge n(k)$ and if $0<t<1$, then \cite[Lemma 1.8.1]{G94} implies:
$$
||e^{-tD_M}-\sum_{n=0}^{n(k)}E_n(t,D_M)||_{-k,k}\le C_kt^k
$$
This gives rise to a corresponding estimate (after increasing $n(k)$ appropriately):
\begin{equation}\label{eqn-2.i}
||K(t,x,\tilde x,D_M)-\sum_{n=0}^{n(k)}K_n(t,x,\tilde x,
D_M)||_{C^k}\le C_kt^k\,.
\end{equation}
\subsection{Examining the heat content}\label{sect-2.5}
We use Equation~(\ref{eqn-2.g}), Equation~(\ref{eqn-2.h}), and Equation~(\ref{eqn-2.i}) to expand
\begin{eqnarray*}
&&\beta(\phi,\rho,D_M)(t)=\sum_{2j-|\alpha|=0}^{n(k)}(2\pi)^{-m}\frac{t^j}{j!}\iiint
e^{-t|\xi|_{g^*(x)}^2-\sqrt{-1}(x-\tilde x)\cdot\xi}\xi^\alpha\\
&&\qquad\qquad\qquad\qquad\times \langle
r_{n,j,\alpha}(x,D_M)\phi(x),\rho(\tilde x)\rangle
 d\nu_xd\nu_\xi d\nu_{\tilde x}+O(t^k)\,.
\end{eqnarray*}
We examine a typical term in the sum setting:
\begin{eqnarray*}
&&\beta_{n,j,\alpha}(\phi,\rho)(t):=(2\pi)^{-m}\frac{t^j}{j!}\iiint e^{-t|\xi|_{g^*(x)}^2-\sqrt{-1}(x-\tilde x)\cdot\xi}\xi^\alpha\\
&&\qquad\qquad\qquad\qquad\times\langle r_{n,j,\alpha}(x)\phi(x),\rho(\tilde x)\rangle d\nu_xd\nu_\xi d\nu_{\tilde x}\,.
\end{eqnarray*}
Here all integrals are over $\mathbb{R}^m$ and converge absolutely for $t>0$; $\phi$ and $\rho$ have compact support.
We change variables setting $\tilde\xi:=t^{1/2}\xi$ to
express:
\begin{eqnarray*}
&&\beta_{n,j,\alpha}(\phi,\rho)(t)=\frac{t^{j-\frac12m-\frac12|\alpha|}}{j!}(2\pi)^{-m}
\iiint e^{-|\tilde\xi|_{g^*(x)}^2-\sqrt{-1}(x-\tilde x)\cdot\tilde\xi/\sqrt{t}}\tilde\xi^\alpha\\
&&\qquad\qquad\qquad\qquad\times\displaystyle \langle
r_{n,j,\alpha}(x,D_M)\phi(x),\rho(\tilde x)\rangle
d\nu_xd\nu_{\tilde\xi}d\nu_{\tilde x}\,.\nonumber
\end{eqnarray*}
Note that $\frac12n=j-\frac12|\alpha|$. We adopt the notation of Lemma~\ref{lem-2.1} and
make a complex change of coordinates
setting:
$$\textstyle\eta=\tilde\xi+\frac12\sqrt{-1}(X-\tilde X)/\sqrt t\,.$$
We then apply Lemma~\ref{lem-2.1} and the binomial theorem to express:
\begin{eqnarray*}
&&\beta_{n,j,\alpha}(\phi,\rho,)(t)\\
&=&(2\pi)^{-m}\sum_{\alpha_1+\alpha_2=\alpha}
{\textstyle\frac{\alpha!}{j!\alpha_1!\alpha_2!}}(-\sqrt{-1})^{|\alpha_2|}
t^{(n-m)/2}\iiint e^{-||\eta||_{g(x)}^2}\eta^{\alpha_1}\\
&&\qquad\times e^{-||x-\tilde
x||_{g(x)}^2/(4t)}\left(\textstyle\frac{X-\tilde X}{2\sqrt
t}\right)^{\alpha_2} \langle
r_{n,j,\alpha}(x,D_M)\phi(x),\rho(\tilde x)\rangle d\nu_\eta
d\nu_xd\nu_{\tilde x}\,.\nonumber
\end{eqnarray*}
The $d\nu_\eta$ integral is over the complex domain
$\eta\in\mathbb{R}+\frac12\sqrt{-1}\textstyle\frac{X-\tilde X}{\sqrt t}$.
But we can deform that domain back to the real domain $\eta\in\mathbb{R}$. Set
\begin{eqnarray*}
&&c_{\alpha_1,\alpha_2,j}:=(2\pi)^{-m}\frac1{j!}\frac{(\alpha_1+\alpha_2)!}{\alpha_1!\alpha_2!}
(-\sqrt{-1})^{|\alpha_2|}
\int\eta^{\alpha_1}e^{-|\eta|_{g^*(x)}^2}d\nu_\eta\text{ to express}\\
&&\beta_{n,j,\alpha}(t)=\sum_{\alpha_1+\alpha_2=\alpha}c_{\alpha_1,\alpha_2,j}t^{(n-m)/2}\\
&&\qquad\times \iint e^{-||x-\tilde
x||_{g(x)}^2/(4t)}\left(\textstyle\frac{X-\tilde X}{2\sqrt
t}\right)^{\alpha_2}\langle
r_{n,j,\alpha}(x,D_M)\phi(x),\rho(\tilde x)\rangle
d\nu_xd\nu_{\tilde x}\,.
\end{eqnarray*}
This sum ranges over $|\alpha_1|$ even as otherwise $c_{\alpha_1,\alpha_2}$ vanishes. Thus
$|\alpha_2|\equiv|\alpha|\equiv n\mod 2$. This reduces the proof to considering expressions of the form:
\begin{eqnarray}
f_{n,j,\alpha,\alpha_2}(t)&:=&t^{(n-m)/2}\iint
e^{-||x-\tilde x||_{g(x)}^2/(4t)}(\textstyle\frac{X-\tilde
X}{\sqrt t})^{\alpha_2}\nonumber\\
&&\qquad\times \langle r_{n,j,\alpha}(x,D_M)\phi(x),\rho(\tilde
x)\rangle d\nu_xd\nu_{\tilde x},
\label{eqn-2.j}\\
&&\text{ where }|\alpha_2|\equiv n\text{ mod }2\text{ and
}\operatorname{ord}(r_{n,j,\alpha}(x,D_M))=n\,.\nonumber
\end{eqnarray}
\subsection{Asymptotic series}\label{sect-2.6}

Before proceeding further with our analysis of Equation~(\ref{eqn-2.j}), we must establish the existence of asymptotic
series in certain quite general contexts:
\begin{lemma}\label{lem-2.3}
Let $\Phi\in L^1(\mathbb{R}^m)$, let $\rho\in
C^\infty(\mathbb{R}^m)$ have compact support in an open subset
$\mathcal{O}\subset\mathbb{R}^m$, and let $(X-\tilde
X)_i:=g_{ij}(x-\tilde x)^j$. Let \medbreak $\displaystyle
F(t):=t^{(n-m)/2}\int_{\mathcal{O}}\int_{\mathcal{O}}
e^{-||x-\tilde x||^2_{g(x)}/(4t)}\left({\textstyle\frac{X-\tilde
X}{\sqrt t}}\right)^{\alpha_2} \langle\Phi(x),\rho(\tilde
x)\rangle d\nu_xd\nu_{\tilde x},\text{ if }m\ge1$, \medbreak
$\displaystyle G(t):=t^{(n-1)/2}\int_0^{\infty}\int_0^{\infty}
e^{-|x+\tilde x|^2_e/(4t)}\left( {\textstyle\frac{X+\tilde
X}{\sqrt t}}\right)^{\alpha_2}\langle\Phi(x),\rho(\tilde x)\rangle
d\nu_xd\nu_{\tilde x},\text{ if }m=1$.
\begin{enumerate}
\item There exist smooth coefficients $c_{\sigma,\alpha_2}=c_{\sigma,\alpha_2}(g(x))$
so that there is a complete asymptotic series as $t\downarrow0^+$ of the form
$$F(t)\sim\sum_{|\sigma|=0}^\infty t^{(n+|\sigma|)/2}\int_{\mathcal{O}} c_{\sigma,\alpha_2}
\langle \Phi(x),\rho^{(\sigma)}(x)\rangle d\nu_x\,.$$
\item Let $m=1$. Near $0$, we suppose $\Phi\sim x^{-\alpha}\sum_{i\ge0}C_ix^i$ for $\Re(\alpha)<1$.
There exist universal constants $c_{i,j,\alpha_2}$ so that
 there is a complete asymptotic series as $t\downarrow0^+$ of the form
$$G(t)\sim\sum_{i,j=0}^\infty t^{(n+1+i+j-\alpha)/2}c_{i,j,\alpha_2}
\langle C_i,\rho^{(j)}(0)\rangle\,.$$
\end{enumerate}\end{lemma}

\begin{proof} We make the change of variables $\tilde x=x+u$ and dually $\tilde X=X+U$ where $U_i=g_{ij}u^j$ to express
$$
F(t)=t^{(n-m)/2}\iint
e^{-||u||_{g(x)}^2/(4t)}\textstyle\left(\frac{U}{\sqrt
t}\right)^{\alpha_2}\langle\Phi(x),\rho(x+u)\rangle
d\nu_ud\nu_x\,.
$$
The $d\nu_u$ integral decays exponentially for $|u|>t^{1/4}$ so we may assume the $d\nu_u$ integral is localized
to $|u|<t^{1/4}$. For $u$ small, we use Equation~(\ref{eqn-2.a}) to express:
\begin{eqnarray*}
&&\rho(x+u)\sim \sum_{|\sigma|\le N}\textstyle\frac1{\sigma!}u^\sigma d_x^\sigma\rho(x)+O(u^N),\\
&&F(t)=t^{(n-m)/2}\sum_{|\sigma|\le
N}{\textstyle\frac1{\sigma!}}\iint e^{-||u||_{g(x)}^2/(4t)}\\
&&\qquad\times
\left\{ ({\textstyle\frac{U}{\sqrt t}})^{\alpha_2}
u^\sigma\langle \Phi(x),\rho^{(\sigma)}(x)\rangle+O(|u|^N)\right\}
d\nu_ud\nu_x\,.\nonumber
\end{eqnarray*}
We set $\tilde u=u/\sqrt t$ and $\tilde U=U/\sqrt{t}$ to express
$$
F(t)\sim\sum_{|\sigma|\le N}{\textstyle\frac1{\sigma!}}t^{(n+|\sigma|)/2}
\iint e^{-||\tilde u||_{g(x)}^2/4}\tilde U^{\alpha_2}\tilde u^\sigma\langle\Phi(x),\rho^{(\sigma)}(x)\rangle
d\nu_{\tilde u}d\nu_x\,.
$$
The $d\nu_x$ integral remains an integral over $\mathcal{O}$. But as $t\downarrow0$, the $d\nu_{\tilde u}$
integral expands to $\mathbb{R}^m$ and defines the coefficients $c_{\sigma,\alpha_2}=c_{\sigma,\alpha_2}(g)$.
This establishes Assertion~(1).

Let $m=1$. We note that $G$ decays exponentially
for $x\ge\varepsilon>0$ or $\tilde x\ge\varepsilon>0$. On the small square, we expand
$$\Phi(x)\sim\sum_{i\ge0}C_ix^i\text{ and }\rho(\tilde x)\sim\tilde x^j\rho^{(j)}(0)\,.$$
We then make the change of variables with $u=x/\sqrt t$ and
$\tilde u=\tilde x/\sqrt t$ to express \medbreak\qquad
$$ G(t)\sim t^{(n+1)/2}\sum_{i+j\le
N}t^{(i+j-\alpha)/2}c_{i,j,\alpha}\langle
C_i,\rho^j(0)\rangle$$
where
\medbreak\qquad\quad\hfill{$\displaystyle c_{i,j,\alpha}:=\int_0^\infty\int_0^\infty
e^{-\frac14||u+\tilde u||^2_e}u^{i-\alpha}\tilde u^{j}
d\nu_ud\nu_{\tilde u}$.}\hfill\phantom{.}\end{proof}

\subsection{The interior terms in Theorem~\ref{thm-1.2}}\label{sect-2.7}
We apply Lemma~\ref{lem-2.3}~(1) to the case
$\Phi=r_{n,j,\alpha}\phi$ in Equation~(\ref{eqn-2.j}). By
assumption $r_{n,j,\alpha}$ is of order $n$ in the derivatives of
the total symbol of $D_M$. We have $\rho^{(\sigma)}$ is of order
$|\sigma|$ in the derivatives of $\rho$. Thus we have expressions
which are of order $n+|\sigma|$ in the derivatives of the symbol
of $D_M$ and in the derivatives of $\rho$. Furthermore, the
$d\nu_{\tilde u}$ integral vanishes unless $|\sigma|+|\alpha_2|$
is even. Since $|\alpha_2|\equiv n$ mod $2$, this implies
$|\sigma|+n$ is even so terms involving fractional powers of $t$
vanish as claimed. This leads to exactly the sort of interior
expansion described in Theorem~\ref{thm-1.2}.

\subsection{The heat content on a chart near the boundary of $\Omega$}\label{sect-2.8}
We now assume the coordinate chart meets the boundary. Again, we examine Equation~(\ref{eqn-2.j}).
We set $x=(r,y)$; the $d\nu_r$ integral ranges over $0\le r<\infty$
and the $d\nu_y$ integral ranges over $y\in\mathbb{R}^{m-1}$.
The $dy$ and $d\tilde y$ integrals are handled using the analysis of Lemma~\ref{lem-2.3}~(1).
We therefore suppress these variables and
concentrate on the $d\nu_r$ integrals and in essence assume that we are dealing with a $1$-dimensional problem;
we can always choose the coordinates so $ds^2=dr^2+g_{ab}(r,y)dy^idy^j$. We resume
the computation with Equation~(\ref{eqn-2.j}) where we do not perform the integrals in the two variables normal to the boundary.
We suppress other elements of the notation to examine an integral of the form:
$$f(t):=t^{(n-1)/2}
\int_{x=0}^\infty\int_{\tilde x=0}^\infty  e^{-||x-\tilde
x||_e^2/(4t)}(\textstyle\frac{X-\tilde X}{\sqrt t})^{\alpha_2}
\langle r_{n,j,\alpha}(x,D_M)\phi(x),\rho(\tilde x)\rangle
d\nu_xd\nu_{\tilde x}\,.
$$
Here $q$ has compact support in $(x,\tilde x)$ and is homogeneous
of degree $n$ in the derivatives of the symbol of $D_M$, in the
derivatives of $\phi$, and in the derivatives of $\rho$; there is
no trouble with convergence. We suppress the role of $|_{g}$ in
the tangential integrals which can also depend on the normal
parameter. A crucial point is that the extra power of
``$-\frac12$" occurs in applying Lemma~\ref{lem-2.3} to
$\mathbb{R}^{m-1}$. We set
\begin{eqnarray*}
&& f_1(t):=t^{(n-1)/2}\int_{x=0}^\infty\int_{\tilde x =-\infty}^0
e^{-||x-\tilde x||_e^2/(4t)}\\
&&\qquad\qquad\quad\times(\textstyle\frac{X-\tilde X}{\sqrt
t})^{\alpha_2} \langle r_{n,j,\alpha}(x, D_M)\phi(x),\rho(\tilde
x)\rangle d\nu_xd\nu_{\tilde x}.
\end{eqnarray*}
 Again, there is no trouble with convergence.
The sum $f(t)+f_1(t)$ can then be handled as in
Section~\ref{sect-2.7} and gives rise to the interior term we have
been studying. Thus everything new comes from $f_1(t)$ and this is
handled by Lemma~\ref{lem-2.3}~(2) with $\alpha=0$ after we
replace $\tilde x$ by $-\tilde x$. The terms multiplying
$t^{(n+1+|\alpha_1|+|\alpha_2|)/2}$ have degree
$n+|\alpha_1|+|\alpha_2|$ in the derivatives of the symbol of
$D_M$, of the derivatives of $\phi$, and of the derivatives of
$\rho$. After setting $j=n+|\alpha_1|+|\alpha_2|$ and summing, we
obtain the boundary terms of Theorem~\ref{thm-1.2}. We start out
 at $t^{(n-1)/2}$ but then we have two factors of $t^{1/2}$
 arising from the $x$ and $\tilde x$ change of variable.

 \begin{remark}\rm It is clear from the construction that the coefficients in the boundary asymptotic
 expansion depend holomorphically on the complex parameter $\alpha$ for $\Re(\alpha)<1$; the fact that the
 constants $\varepsilon_{\nu,\alpha}$ of Lemma~\ref{lem-1.5} are holomorphic as well now follows.
 \end{remark}

\section{The proof of Theorem~\ref{thm-1.6}}\label{sect-3}
We will establish Theorem~\ref{thm-1.6} by evaluating the universal constants which appear
in Lemma~\ref{lem-1.5}.
In Section~\ref{sect-3.1}, we establish Lemma~\ref{lem-1.4} which relates to the heat content
asymptotics on the line.
This result is then used
in Section~\ref{sect-3.2} to determine the constants
$\{\varepsilon_{0,\alpha},\varepsilon_{1,\alpha},\varepsilon_{3,\alpha},\varepsilon_{4,\alpha},
\varepsilon_{7,\alpha},\varepsilon_{14,\alpha}\}$. Then in Section~\ref{sect-3.3}, we use product formulas
to determine $\{\varepsilon_{6,\alpha},\varepsilon_{12,\alpha},\varepsilon_{13,\alpha}\}$.
We complete the computation in Section~\ref{sect-3.4} using warped products.

\subsection{The proof of Lemma~\ref{lem-1.4}}\label{sect-3.1}
We apply the analysis of Section~\ref{sect-2} to the 1-dimensional setting. We work in the scalar
setting and set $D=-\partial_x^2$. Consequently
\begin{eqnarray*}
&&p_2(x,\xi)=\xi^2,\quad p_1(x,\xi)=0,\quad p_0(x)=0,\\
&&r_0(x,\xi;\lambda)=(\xi^2-\lambda)^{-1},\quad r_n(x,\xi;\lambda)=0\text{ for }n\ge1,\\
&&e_0(x,\xi;t)=e^{-t\xi^2},\quad e_n(x,\xi;t)=0\text{ for }n\ge 1\,.
\end{eqnarray*}
Consequently we have $K_n=0$ for $n\ge1$ while
\begin{eqnarray*}
K_0(x,\tilde x;t)&=&\frac1{2\pi}\int_{-\infty}^\infty e^{-\sqrt{-1}(x-\tilde x)\cdot\xi}e^{-t\xi^2}d\xi\\
&=&\frac1{2\pi}e^{-(x-\tilde x)^2/(4t)}\int_{-\infty}^\infty e^{-t|\xi|^2}d\xi\\
&=&\frac1{\sqrt{4\pi t}}e^{-(x-\tilde x)^2/(4t)}\,.
\end{eqnarray*}
This is, of course, not surprising as this is the heat kernel in flat space. Let
\begin{eqnarray*}
&&f_1(t):=\frac1{\sqrt{4\pi t}}\int_{x=0}^\infty\int_{\tilde x=-\infty}^\infty
e^{-(x-\tilde x)^2/(4t)}\phi(x)\rho(\tilde x)d\tilde x dx,\\
&&f_2(t):=\frac1{\sqrt{4\pi t}}\int_0^\infty\int_0^\infty e^{-(x+\tilde x)^2/(4t)}\phi(x)\rho(-\tilde
x)d\tilde xdx\,.
\end{eqnarray*}
We may then express $\beta(\phi,\rho,D_M)(t)=f_1(t)-f_2(t)$. We
change variables setting $\tilde x=x+u$ to express
$$
f_1(t)\sim\frac1{\sqrt{4\pi
t}}\int_{x=0}^\infty\int_{u=-\infty}^\infty\sum_{k=0}^\infty \frac
1{k!}\phi(x)\rho^{(k)}(x)u^k e^{-u^2/(4t)}dudx\,.
$$
In the sum, we must have $k=2\bar k$ is even.
We integrate by parts $\bar k$ times to evaluate the constants which appear. Alternatively we change variables $u^2=4tv$
and use standard formulae for the $\Gamma$-function to obtain that

\begin{eqnarray*}
&&\frac1{\sqrt{4\pi t}}\int_{-\infty}^\infty\frac1{(2\bar
k)!}u^{2\bar k}e^{-u^2/(4t)}du
=\frac2{\sqrt{4\pi t}(2\bar k)!}\int_0^\infty u^{2\bar k}e^{-u^2/(4t)}du\\
&=&\frac{2^{2\bar k}t^{\bar k}}{\sqrt{\pi}(2\bar k)!}\int_0^{\infty}v^{\bar k-\frac12}e^{-v}dv
=\frac{2^{2\bar k}\Gamma(k+\frac12)t^{\bar k}}{\sqrt{\pi}(2\bar
k)!}=\frac{t^{\bar k}}{\bar k!}.
\end{eqnarray*}
The interior terms arise from expanding
$$f_1(t)\sim\sum_{\bar k}\frac{t^{\bar k}}{\bar k!}\int_{0}^\infty \phi(x)\rho^{(2\bar k)}(x)dx
\sim\sum_{\bar k}(-1)^{\bar k}\frac{t^{\bar k}}{\bar k!}\int_0^\infty\phi(x)D^{\bar k}\rho(x)dx\,.$$

Next we evaluate $f_2$ (and we have to subtract this term). We expand
$$\phi(x)\sim x^{-\alpha}\sum_i\phi_ix^i\text{ and }\rho(\tilde x)\sim\sum_j\rho_j\tilde x^j\,.$$
We do not put in the factorials so $\rho_j=\frac1{j!}\rho^{(j)}(0)$.
$$
f_2(t)\sim\frac1{\sqrt{4\pi t}}\int_{0}^\infty\int_0^\infty e^{-(x+\tilde x)^2/(4t)} \sum_{i,j}
\phi_i\rho_jx^{i-\alpha}\tilde x^jd\nu_xd\nu_{\tilde x} \,.$$ We
change variables to set $x=\sqrt t u$ and $\tilde x=\sqrt t\tilde
u$ to complete the proof of Lemma~\ref{lem-1.4} by expressing
\medbreak $\displaystyle
f_2(t)\sim\frac1{\sqrt{4\pi}}\sum_{i,j}t^{(i+j-\alpha+1)/2}(-1)^j\phi_i\rho_j
\times \int_{0}^\infty\int_{0}^\infty e^{-(u+\tilde
u)^2/4}u^{i-\alpha}\tilde u^jdud\tilde u$.\hfill\qed

\subsection{Evaluating the constants for the $1$-dimensional case}\label{sect-3.2}
We use e.g. Mathematica \cite{M} to compute the coefficients of
Lemma~\ref{lem-2.1}: \medbreak $\displaystyle
\varepsilon_{0,\alpha}=-\frac1{\sqrt{4\pi}}\int_0^\infty\int_0^\infty
x^{-\alpha}e^{-(x+y)^2/4}
dxdy=\frac{2^{1-\alpha}\Gamma(\frac{2-\alpha}2)}{(-1+\alpha)\sqrt{4\pi}}=\frac{c_\alpha}2$
\medbreak $\displaystyle
\varepsilon_{1,\alpha}=-\frac1{\sqrt{4\pi}}\int_0^\infty\int_0^\infty
x^{1-\alpha}e^{-(x+y)^2/4}
dxdy=\frac{2^{2-\alpha}\Gamma(\frac{3-\alpha}2)}{(-2+\alpha)\sqrt{4\pi}}=\frac{c_{\alpha-1}}2$
\medbreak $\displaystyle
\varepsilon_{3,\alpha}=+\frac1{\sqrt{4\pi}}\int_0^\infty\int_0^\infty
x^{-\alpha}ye^{-(x+y)^2/4}
dxdy=-\frac{2^{1-a}\Gamma(\frac{1-\alpha}2)}{(-2+\alpha)\sqrt{4\pi}}=\frac{-c_{\alpha-1}}{2(1-\alpha)}$
\medbreak
$\displaystyle\varepsilon_{4,\alpha}=-\frac1{\sqrt{4\pi}}\int_0^\infty\int_0^\infty
x^{2-\alpha}e^{-(x+y)^2/4}
dxdy=\frac{2^{3-\alpha}\Gamma(\frac{4-\alpha}2)}{(-3+\alpha)\sqrt{4\pi}}=\frac{c_{\alpha-2}}2$
\medbreak $\displaystyle
\varepsilon_{14,\alpha}=+\frac1{\sqrt{4\pi}}\int_0^\infty\int_0^\infty
x^{1-\alpha}ye^{-(x+y)^2/4}
dxdy=-\frac{2^{2-a}\Gamma(\frac{2-\alpha}2)}{(-3+\alpha)\sqrt{4\pi}}=\frac{-c_{\alpha-2}}{2(2-\alpha)}$
\medbreak $\displaystyle
\varepsilon_{7,\alpha}=-\frac1{\sqrt{4\pi}}\int_0^\infty\int_0^\infty
x^{-\alpha}y^2e^{-(x+y)^2/4}
dxdy=-\frac{2^{3-\alpha}\Gamma(\frac{2-\alpha}2)}{(3-4\alpha+\alpha^2)\sqrt{4\pi}}$
\medbreak\qquad\qquad\qquad
$\displaystyle=\frac{c_{\alpha-2}}{(2-3\alpha+\alpha^2)}$.

\subsection{Product formulae}\label{sect-3.3} We evaluate
$\{\varepsilon_{6,\alpha},\varepsilon_{12,\alpha},\varepsilon_{13,\alpha}\}$.
Let $(N,g_N)$ be a closed Riemannian manifold, let $D_N$ be the scalar Laplacian on $N$,
let $S=(S^1,dx^2)$, and let $D_S=-\partial_x^2$ where $x$ is the usual periodic angular parameter on
the circle $S^1$.
We consider the product manifold.
Let
$$\begin{array}{ll}
(M,g_M):=(N\times S^1,g_N+dx^2),&D_M:=D_N+D_S,\\
\Sigma:=[0,\pi],&\Omega=N\times\Sigma,\\
\phi(x,z)=\phi_\Sigma(x)\phi_N(z),&\rho(x,z)=\rho_\Sigma(x)\rho_N(z)
\end{array}$$
for $\phi_N\in C^\infty(N)$, $\rho_N\in C^\infty(N)$, $\phi_\Sigma\in\mathcal{K}_\alpha(\Sigma)$,
and $\phi_\Sigma\in C^\infty(\Sigma)$.
As the structures decouple,
\begin{eqnarray*}
&&e^{-tD_M}=e^{-tD_S}e^{-tD_N}\text{ so}\\
&&\beta_\Omega(\phi,\rho,D_M)(t)=\beta_\Sigma(\phi,\rho,D_S)(t)\cdot
\beta_N(\phi_N,\rho_N,D_N)(t)\,.
\end{eqnarray*}
There are, of course, no boundary terms in the asymptotic series
for $\beta_N$, and the interior terms are given by
Lemma~\ref{lem-1.4}. Equating the two asymptotic series then
permits us to conclude that
$\beta_{2,\alpha}^{\partial\Omega}(\phi,\rho,D_M)=\mathcal{E}_2+\mathcal{E}_0$
where
\begin{eqnarray*}
&&\mathcal{E}_2:=\int_{\partial\Sigma}\beta_{2,\alpha}^{\partial\Sigma}
(\phi_\Sigma,\rho_\Sigma,D_\Sigma)(x)dx\cdot\int_N\phi_N(z)\rho_N(z)dz\text{ and}\\
&&\mathcal{E}_0:=\frac12c_{\alpha-2}\int_\Sigma \phi_{\Sigma,0}(y)\rho_{\Sigma,0}(y)dy\cdot
\int_N-\phi_N(z)\cdot D_N\rho_N(z)dz\,.
\end{eqnarray*}
We suppress terms not of interest to equate
\begin{eqnarray*}
&&\int_{\partial\Sigma}\int_N\phi_{\Sigma,0}(x)\rho_{\Sigma,0}(x)\\
&&\qquad\cdot
\left\{\varepsilon_{6,\alpha}E_N(z)\phi_N(z)\rho_N(z)+\varepsilon_{12,\alpha}\phi_{N:a}\rho_{N:a}
+\varepsilon_{13,\alpha}\tau\phi_N\rho_N\right\}dxdz\\
&=&
\frac12c_\alpha\int_{\partial\Sigma}\phi_{\Sigma,0}(x)\rho_{\Sigma,0}(x)dx\cdot
\int_N\phi_N(z)(\rho_{N;aa}+E)\rho_N(z)dz\,.
\end{eqnarray*}
Integrating by parts on $N$ permits us to see
$$\int_N\phi_N(z)\rho_{N;aa}(z)dz=-\int_N\phi_{N;a}(z)\rho_{N;a}(z)dz\,.$$
Applying the recursion relation of Equation~(\ref{eqn-1.f}) then yields:
$$
-\varepsilon_{12,\alpha}=\varepsilon_{6,\alpha}=\frac12c_\alpha=-\frac{\alpha-3}{2(\alpha-1)(\alpha-2)}\frac12c_{\alpha-2},\quad\text{ and }
\quad\varepsilon_{13,\alpha}=0\,.
$$

\subsection{Warped products}\label{sect-3.4}
We now determine the coefficients involving $L$ and $\operatorname{Ric}$.
The fact that we are working with quite general operators
is now crucial. We extend the formalism
of Section~\ref{sect-3.3} and consider warped products. Let
$$\Sigma=[0,\pi],\quad S=[0,2\pi]/0\sim2\pi,\quad D_S=-\partial_x^2\,.$$
Let $\mathbb{T}^{m-1}$ be the torus with periodic
parameters $(\theta_1,...,\theta_{m-1})$, let $M=
\mathbb{T}^{m-1}\times S^1$, and let $\Omega=
\mathbb{T}^{m-1}\times\Sigma$. Let $f_a\in C^\infty(S)$ be a
collection of smooth functions satisfying $f_a(0)=0$ and
$f_a\equiv0$ near $x=\pi$. Let $\delta_a\in\mathbb{R}$. Set
$$\begin{array}{ll}
M:=\mathbb{T}^{m-1}\times S,&
\displaystyle ds^2_M=\sum_{a=1}^{m-1}e^{2f_a(x)}d\theta_a\circ d\theta_a+dx\circ dx,\\
\Omega:=\mathbb{T}^{m-1}\times[0,\pi],&\displaystyle
D_M:=-\sum_{a=1}^{m-1}e^{-2f_a(x)}(\partial_{\theta_a}^2+\delta_a\partial_{\theta_a})-\partial_x^2\,.
\end{array}$$
Let $\phi_\Sigma\in\mathfrak{K}_\alpha(\Sigma)$ with $\phi_\Sigma$ vanishing identically near $\pi$ and
let $\rho_\Sigma\in C^\infty([0,\pi])$ with $\rho_\Sigma$ vanishing identically near $\pi$ as well. Set:
$$\phi_\Omega(x,y)=\phi_\Sigma(x)\text{ and }\rho_\Omega(x,y)=\rho_\Sigma(x)e^{-\sum_if_a(x)}\,.$$

\begin{lemma}\label{lem-3.1}
$\beta_{j,\alpha}^{\partial M}(\phi_\Omega,\rho_\Omega,D_M)
=\beta_{j,\alpha}^{\partial\Sigma}(\phi_\Sigma,\rho_\Sigma,D_S)\operatorname{vol}(\mathbb{T}^{m-1})$ for $j\ge0$.
\end{lemma}

\begin{proof} Note that $x$ is
the geodesic distance to $\{0\}$ in $\Sigma$ and that $x$ is the
geodesic distance to $\{0\times\mathbb{T}^{m-1}\}$ in
$\Omega$; the component where $x=\pi$ plays no role as $\phi$ and
$\rho$ vanish identically near this component. Since $\phi_\Omega$
is independent of $y\in\mathbb{T}$, the problem decouples and
$$
\left\{e^{-tD_M}\phi_\Omega\right\}(x,y;t)=\left\{e^{-tD_S}\phi_\Sigma\right\}(x;t)\,.
$$
The Riemannian measure on $M$ takes the form:
$$d\nu_M=\sqrt{\det g_{ij}}dydx=e^{\sum_af_a}dydx\,.$$
Since $\rho_\Omega d\nu_\Omega=\rho_\Sigma dxdy$, we have:
$$
\beta_{\Omega}(\phi_\Omega,\rho_\Omega,D_M)(t)=\beta_\Sigma(\phi_\Sigma,\rho_\Sigma,D_S)(t)\cdot
\operatorname{vol}(\mathbb{T}^{m-1})\,.$$ Lemma \ref{lem-3.1} now
follows for $\alpha\not\in\mathbb{Z}$ since the interior
invariants and the boundary invariants do not interact.
Since the invariants $\beta_{j,\alpha}^{\partial \Omega}$ and
$\beta_{j,\alpha}^{\partial\Sigma}$ are analytic in $\alpha$, the
desired conclusion also follows for $\alpha\in\mathbb{Z}$. Thus
even if one were only interested in the case $\alpha=0$, it is
convenient to have more general values of $\alpha$ available.
\end{proof}

We apply Lemma~\ref{lem-1.5}.
Although the structures are flat on $\Sigma$, they are not flat on $\Omega$ and this makes all the difference.
We determine the relevant tensors as follows:
$$
\begin{array}{ll}
\Gamma_{abm}=-f_a^\prime\delta_{ab}e^{2f_a},&\Gamma_{ab}{}^m=-f_a^\prime e^{2f_a}\delta_{ab},\\
\Gamma_{amb}=f_a^\prime\delta_{ab}e^{2f_a}
,&\Gamma_{am}{}^b=f_a^\prime\delta_{a,b},
   \vphantom{\vrule height 12pt}\\
L_{ab}=\Gamma_{ab}{}^m|_{\partial M}=-f_a^\prime\delta_{ab},&\vphantom{\vrule height 12pt}\\
\textstyle\omega_a={\textstyle\frac12}e^{2f_a}\delta_a,&\tilde\omega_a=-\omega_a
   =-{\textstyle\frac12}e^{2f_a}\delta_a,
   \vphantom{\vrule height 12pt}\\
\omega_m=-{\textstyle\frac12}\sum_af_a^\prime,&
\tilde\omega_m=-\omega_m={\textstyle\frac12}\sum_af_a^\prime\,.\vphantom{\vrule
height 12pt}
\end{array}$$
Consequently: \medbreak\qquad
$R_{ambm}=g((\nabla_a\nabla_m-\nabla_m\nabla_a)e_b,e_m)
=\Gamma_{ac}{}^m\Gamma_{mb}{}^c-\partial_m\Gamma_{ab}{}^m$
\smallbreak\qquad\quad
$=\{-(f_a^\prime)^2+f_a^{\prime\prime}+2(f_a^\prime)^2\}e^{2f_a}\delta_{ab}$,
\medbreak\qquad
$\operatorname{Ric}_{mm}=-\textstyle\sum_a\left\{f_a^{\prime\prime}+(f_a^\prime)^2\right\}$,
\medbreak\qquad $E|_{\partial
M}=-\partial_m\omega_m-\omega_a^2-\omega_m^2+\omega_m\Gamma_{aa}{}^m$
\smallbreak\qquad\quad
$=\textstyle\frac12\sum_af_a^{\prime\prime}-\frac14\sum_a\delta_a^2-\frac14\sum_{a,b}f_a^\prime
f_b^\prime+\frac12\sum_{a,b}f_a^\prime f_b^\prime$
\medbreak\qquad\quad
$=\textstyle\frac12\sum_af_a^{\prime\prime}-\frac14\sum_a\delta_a^2+\frac14\sum_{a,b}f_a^\prime
f_b^\prime$.

\medbreak\noindent We introduce the notation $\phi_\Omega$ and $\rho_\Omega$ to emphasize
we are computing with the structures on $M$ and not on $S$. We evaluate on the component $x=0$:
\medbreak\quad
$\phi_{\Omega,0}|_{x=0}=\phi_{\Sigma,0}(0)$,
\medbreak\quad
$\phi_{\Omega,1}|_{x=0}=\{\nabla_{\partial x}(\phi_{\Sigma,0}+x\phi_{\Sigma,1})\}|_{x=0}
=\{(\partial_x-\textstyle{\textstyle\frac12}\sum_af_a^\prime)(\phi_{\Sigma,0}+x\phi_{\Sigma,1})\}|_{x=0}$
\smallbreak\qquad\qquad\phantom{..}
$=-{\textstyle\frac12}\sum_af_a^\prime\phi_{\Sigma,0}(0)+\phi_{\Sigma,1}(0)$,
\medbreak\quad
$\textstyle\phi_{\Omega,2}|_{x=0}={\textstyle\frac12}\{(\nabla_{\partial_x})^2
(\phi_{\Sigma,0}+x\phi_{\Sigma,1}+x^2\phi_{\Sigma,2})\}|_{x=0}$
\smallbreak\qquad\qquad\phantom{..}
$={\textstyle\frac12}\{(\partial_x
-\textstyle{\textstyle\frac12}\sum_af_a^\prime)^2(\phi_{\Sigma,0}+x\phi_{\Sigma,1}+x^2\phi_{\Sigma,2})\}|_{\partial\Sigma}$
\smallbreak\qquad\qquad\phantom{.}
$=\{{\textstyle\frac18}\sum_{a,b}f_a^\prime f_b^\prime
-{\textstyle\frac14}\sum_af_a^{\prime\prime}\}\phi_{\Sigma,0}(0)-\frac12\sum_af_a^\prime\phi_{\Sigma,1}(0)
+\phi_{\Sigma,2}(0)$,
\medbreak\quad
$\rho_{\Omega,0}|_{x=0}=\rho_{\Sigma,0}(0)$,
\medbreak\quad
$\rho_{\Omega,1}|_{x=0}=\{\tilde\nabla_{\partial x}(\rho)\}|_{x=0}
=\{(\partial_x+{\textstyle\frac12}\sum_af_a^\prime)(e^{-\sum_af_a})
(\rho_{\Sigma,0}+x\rho_{\Sigma_1})\}|_{x=0}$
\smallbreak\qquad\qquad\phantom{..}
$=-{\textstyle\frac12}\sum_af_a^\prime\rho_{\Sigma,0}(0)+\rho_{\Sigma,1}(0)$,
\smallbreak\quad
$\rho_{\Omega,2}|_{x=0}={\textstyle\frac12}\{(\tilde\nabla_{\partial
r})^2\rho_\Sigma\}|_{x=0} ={\textstyle\frac12}\{(\partial_x
+{\textstyle\frac12}\sum_af_a^\prime)^2(e^{-\sum_af_a})\}|_{x=0}$
\smallbreak\qquad\qquad\phantom{.}
$=\{\textstyle{\textstyle\frac18}\sum_{.a,b}f_a^\prime f_b^\prime
-{\textstyle\frac14}\sum_af_a^{\prime\prime}\}\rho_{\Sigma,0}(0)
-{\textstyle\frac12}\sum_af_a^\prime\rho_{\Sigma,1}(0)+\rho_{\Sigma,2}(0)$,
\medbreak\quad
$\phi_{\Omega,0:a}\rho_{\Omega,0:a}=-\frac14\sum_a\delta_a^2$.
\medbreak\noindent
The structures defined by $f_a^\prime$, $f_a^{\prime\prime}$, and $\delta_a$ do not appear
in $\beta_{j,\alpha}^{\partial\Sigma}$ and thus these terms
must give zero in $\beta_{j,\alpha}^{\partial\Omega}$. {By considering the monomial
$\sum_{a,b}f_a^\prime f_b^\prime\phi_{\Sigma,0}(0)\rho_{\Sigma,0}(0)$
in $\beta_{2,\alpha}^{\partial M}$, we obtain the relation
$$\textstyle\frac18\varepsilon_{4,\alpha}
+{\textstyle\frac12}\varepsilon_{5,\alpha}+{\textstyle\frac14}\varepsilon_{6,\alpha}
+{\textstyle\frac18}\varepsilon_{7,\alpha}
+{\textstyle\frac12}\varepsilon_{8,\alpha}
+\varepsilon_{10,\alpha}{+\frac14\varepsilon_{14,\alpha}}=0\,.$$
We obtain other relations by considering suitable monomials:
$$\begin{array}{ll}
\text{Relation}&\text{Monomial}\pbg\\
-{\textstyle\frac12}\varepsilon_{1,\alpha}-\varepsilon_{2,\alpha}-{{\textstyle\frac12}\varepsilon_{3,\alpha}}=0,
&\sum_af_a^\prime\phi_{\Sigma,0}(0)\rho_{\Sigma,0}(0),\pbg\\
-{\textstyle\frac14}(\varepsilon_{6,\alpha}+\varepsilon_{12,\alpha})=0,&
\sum_a\delta_a^2\phi_{\Sigma,0}(0)\rho_{\Sigma,0}(0),\pbg\\
\textstyle-\varepsilon_{9,\alpha}+\varepsilon_{11,\alpha}=0,&
\sum_a(f_a^\prime)^2\phi_{\Sigma,0}(0)\rho_{\Sigma,0}(0),\pbg\\
-{\textstyle\frac12}\varepsilon_{4,\alpha}-\varepsilon_{5,\alpha}-{{\textstyle\frac12}\varepsilon_{14,\alpha}}=0,&
\sum_af_a^\prime\phi_{\Sigma,1}(0)\rho_{\Sigma,0}(0),\pbg\\
-{\textstyle\frac12}\varepsilon_{14,\alpha}-\varepsilon_{8,\alpha}-{{\textstyle\frac12}\varepsilon_{7,\alpha}}=0,&
\sum_af_a^\prime\phi_{\Sigma,0}(0)\rho_{\Sigma,1}(0),\pbg\\
-{\textstyle\frac14}\varepsilon_{4,\alpha}
+{\textstyle\frac12}\varepsilon_{6,\alpha}-{\textstyle\frac14}\varepsilon_{7,\alpha}
-\varepsilon_{9,\alpha}=0,&
\sum_af_a^{\prime\prime}\phi_{\Sigma,0}(0)\rho_{\Sigma,0}(0)\,.\pbg
\end{array}$$
Theorem~\ref{thm-1.6} follows from these equations and the relations established previously.}

\section{Further functorial properties}\label{sect-4}
In Section~\ref{sect-4.1}, we examine dimension shifting and in Section~\ref{sect-4.2},
we relate the Dirichlet and Neumann
heat content asymptotics to the asymptotics we have been studying. In addition to providing
useful crosschecks on our work, these properties are worth noting as they promise to be useful
in other contexts.

\subsection{Dimension shifting}\label{sect-4.1}
Let $\Re(\alpha)<<0$.
If $\phi\in\mathcal{K}_{\alpha-1}$, then we may regard $\phi$ as defining an element
$\tilde\phi\in\mathcal{K}_{\alpha}$. If we expand $\phi=r^{-\alpha+1}(\phi_0+r\phi_1+...)$,
then $\tilde\phi=r^{-\alpha}(0+r\phi_0+r^2\phi_1+...)$. Consequently $\tilde\phi_i=\phi_{i-1}$ for $i\ge1$
and
$$\beta_{j,\alpha}^{\partial\Omega}(\tilde\phi,\tilde\rho,D_M)
=\beta_{j-1,\alpha-1}^{\partial\Omega}(\phi,\rho,D_M)\,.$$
Examining the formulas of Lemma~\ref{lem-1.5} then yields the
relations:
$$\begin{array}{llll}
\varepsilon_{1,\alpha}=\varepsilon_{0,\alpha-1},&
\varepsilon_{4,\alpha}=\varepsilon_{0,\alpha-2},&
\varepsilon_{5,\alpha}=\varepsilon_{2,\alpha-1},&
\varepsilon_{14,\alpha}=\varepsilon_{3,\alpha-1}.
\end{array}$$
These hold, of course, only if $\Re(\alpha)<0$; we use
analytic continuation to derive the general result. Once again, it
is convenient to have values of $\alpha$ other than $\alpha=0$.

\subsection{The Dirichlet and Neumann heat content asymptotics}\label{sect-4.2}
There is a useful relationship between the heat content asymptotics being studied at present and the ones
studied previously.
\begin{lemma}
Let $(M,g)$ be a closed Riemannian manifold. Let $T$ be an isometric involution of $M$
with $\operatorname{Fix}(T)=N$ a totally geodesic submanifold of co-dimension $1$.
Assume $M-N$ decomposes as the union of two open sub manifolds $M_+\cup M_-$ which
are interchanged by $T$.
Let $\Omega_\pm:=M_\pm\cup N$. Then
$$\beta_\Omega(\phi,\rho,\Delta)(t)=
\textstyle\frac12\{\beta_D(\phi,\rho,\Delta)(t)+\beta_N(\phi,\rho,\Delta)(t)\}\,.$$
\end{lemma}

\begin{proof} We can use the $\mathbb{Z}_2$ involution $T$ to choose a spectral resolution
$$\{\lambda_N,\phi_{N,n}\}_{n=1}^\infty\cup\{\lambda_D,\phi_{D,n}\}_{n=1}^\infty$$
for $L^2(M)$ so $T^*\phi_{N,n}=-\Phi_{N,n}$ and $T^*\phi_{D,n}=\Phi_{D,n}$. Then
$\{\lambda_{N,n},\sqrt{2}\Phi_{N,n}\}_{n=1}^\infty$ is a spectral resolution for the Neumann Laplacian on $\Omega_+$
and $\{\lambda_{D,n},\sqrt{2}\Phi_{D,n}\}_{n=1}^\infty$ is a spectral resolution for the Dirichlet Laplacian on
$\Omega_+$. We compute that
\begin{eqnarray*}
&&e^{-t\Delta_D}\phi=2\sum_ne^{-t\lambda_{N,n}}(\phi,\phi_{N,n})_{L^2}\phi_{N,n},\\
&&e^{-t\Delta_N}\phi=2\sum_ne^{-t\lambda_{D,n}}(\phi,\phi_{D,n})_{L^2}\phi_{D,n},\\
&&e^{-t\Delta}\phi=\sum_ne^{-t\lambda_{N,n}}(\phi,\phi_{N,n})_{L^2}\phi_{N,n}
    +\sum_ne^{-t\lambda_{D,n}}(\phi,\phi_{D,n})_{L^2}\phi_{D,n},\\
&&\qquad=\frac12\{e^{-t\Delta_D}\phi+e^{-t\Delta_N}\phi\}\,.
\end{eqnarray*}
The desired result now follows by taking the inner product with $\rho$ and integrating over the support of
$\rho$ which is $\Omega_+$.
\end{proof}

This is not useful for studying the terms which involve the second fundamental form $L_{ab}$. However, if
we average the remaining terms for Dirichlet and Neumann boundary conditions in
Theorem~\ref{thm-1.8}, we
get the analogous terms in Theorem~\ref{thm-1.6}.

\section{Reduction to the closed setting}\label{sect-5}
In Section~\ref{sect-2}, we used the classic calculus of
pseudo-differential operators depending on a complex parameter which was developed by
Seeley \cite{S66,S69}. That formalism
 is valid only for compact and closed manifolds. In this section, we will derive
 Theorem~\ref{thm-1.3} where $(M,g)=(\mathbb{R}^m,g_e)$ or where $(M,g)$
 is a compact subset of $\mathbb{R}^m$ of dimension $m$ from Theorem~\ref{thm-1.2}
 which dealt with compact manifolds without boundary.

 We adopt the following notational conventions.
 Let $\Omega\subset\operatorname{Interior}\{\tilde M\}\subset\tilde M\subset M$
where $\Omega$ and $\tilde M$ are compact manifolds of dimension $m$ with smooth boundaries.
Let  $\epsilon:=dist_g(\partial\Omega,\partial \tilde M)>0$. Let $\beta_\Omega^{\tilde M}$ be
the heat content of $\Omega$ in $\tilde M$ and let $\beta_\Omega^M$ be the heat content of
$\Omega$ in $M$.

\begin{theorem}\label{thm-5.1}
Assume that $(M,g)$ is complete with non-negative Ricci curvature.
Let $\rho$ be continuous on $M$ and let $\phi\in L^1(\Omega)$. Then:
$$
|\beta_\Omega^{M}(\phi,\rho,\Delta_g)(t)-
\beta_\Omega^{\tilde M}(\phi,\rho,\Delta_g)(t)|\le 2^{(2+m)/2}
\lVert\phi\rVert_{L^1(\Omega)}\lVert\rho\rVert_{L^{\infty}(\Omega)}e^{-\epsilon^2/(8t)}\,.
$$
\end{theorem}

\begin{proof}
Let $\tilde K$ be the Dirichlet heat kernel for $\tilde M$.
By minimality,
$$0\le \tilde K(x,\tilde x;t)\le K(x,\tilde
x;t)$$ for all $x\in \tilde M,\tilde x\in \tilde M, t>0$. Since $M$
is stochastically complete we have that
$$
1= \int_MK(x,\tilde x;t)d\tilde x.
$$
Moreover since $M$ and hence $\tilde M$ have non-negative Ricci
curvature we have  by Theorem 3.5.3 in \cite{Hsu} (see also Lemma
5 in \cite{B}) that
\begin{align}\label{eqn-5.a}
\int_{\tilde M}\tilde K(x,\tilde x;t)d\tilde x&\ge
1-2^{(2+m)/2}e^{-\textup{dist}_g(x,\partial \tilde
M)^2/(8t)}\nonumber \\ &=\int_M K(x,\tilde x;t)d\tilde
x-2^{(2+m)/2}e^{-\textup{dist}_g(x,\partial \tilde
M)^2/(8t)}\nonumber \\ &\ge \int_M K(x,\tilde x;t)d\tilde
x-2^{(2+m)/2}e^{-\epsilon^2/(8t)},\ \ x\in \Omega,\ t>0.
\end{align}
So \medbreak\qquad
$\displaystyle|\beta_\Omega(\phi,\rho,\Delta_g)(t)-\tilde
\beta_\Omega(\phi,\rho,\Delta_g)(t)|$ \medbreak\qquad\quad
$\displaystyle=\left|\int_{\Omega}\int_{\Omega}(K(x,\tilde
x;t)-\tilde K(x,\tilde x;t))\phi(x)\rho(\tilde x)dxd\tilde
x\right|$ \medbreak\qquad\quad$\displaystyle \le
\lVert\rho\rVert_{L^{\infty}(\Omega)}\int_{\Omega}\int_{\Omega}(K(x,\tilde
x;t)-\tilde K(x,\tilde x;t))|\phi(x)|dxd\tilde x\nonumber $
\medbreak\qquad\quad$\displaystyle\le
\lVert\rho\rVert_{L^{\infty}(\Omega)}\int_{\Omega}|\phi(x)|dx\left(\int_{\tilde
M}(K(x,\tilde x;t)-\tilde K(x,\tilde x;t))d\tilde x\right)$
\medbreak\qquad\quad$\displaystyle\le
\lVert\rho\rVert_{L^{\infty}(\Omega)}\int_{\Omega}|\phi(x)|dx\left(\int_{M}K(x,\tilde
x;t)d\tilde x-\int_{\tilde M}\tilde K(x,\tilde x;t)d\tilde
x\right)$. \medbreak\noindent Theorem~\ref{thm-5.1} now follows by
Equation~\eqref{eqn-5.a}.
\end{proof}

Suppose that $\mathcal{N}=(\mathbb{R}^m,g_e)$ or that
$\mathcal{N}$ is a compact subset of dimension $m$ with smooth
boundary in $(\mathbb{R}^m,g_e)$. The following lemma will permit
us to deduce Theorem~\ref{thm-1.2} for $\mathcal{N}$ from the
corresponding assertion for closed ambient manifolds by using
Theorem~\ref{thm-5.1} to localize matters to a small neighbourhood
$\tilde M$ of $\Omega$.

\begin{lemma} Let $(\tilde M,g_e)$ be a compact smooth manifold of dimension $m$
which is contained in $\mathbb{R}^m$. Then $(\tilde M,g_e)$ is isometric to
a compact smooth manifold of a flat $m$-dimensional torus.
\end{lemma}

\begin{proof} Let $B_r(0)$ be the ball of radius $0$ about the origin in $\mathbb{R}^m$.
Since $\tilde M$ is compact, $\tilde M$ is contained in $B_n(0)$ for some positive integer $n$. Let
$\Gamma:=\{2n\mathbb{Z}\}^{m}$ be the rescaled integer lattice and let
$\mathcal{T}^m=\mathbb{R}^m/\Gamma$ be a flat torus.
Then $\tilde M\subset B_n(0)$ embeds isometrically in $\mathcal{T}^m$.
\end{proof}

\end{document}